\documentclass[10pt, a4paper]{article}

\usepackage[utf8]{inputenc}
\usepackage{amsmath}
\usepackage{amssymb}
\usepackage{enumerate}
\usepackage{amsthm}
\usepackage{dsfont}

\usepackage[draft]{todonotes}   

\newtheorem{theorem}{Theorem}

\newtheorem{proposition}[theorem]{Proposition}
\newtheorem{lemma}[theorem]{Lemma}
\newtheorem{assumption}[theorem]{Assumption}
\newtheorem*{assumption*}{Assumption}
\theoremstyle{definition}\newtheorem{remark}[theorem]{Remark}

\newcommand{\R}{\mathbb{R}}

\newcommand{\m}{\mathfrak{m}}
\newcommand{\twp}{\hat{u}}

\title{$L^2$-stability of Traveling Wave Solutions to Nonlocal Evolution Equations}

\author{Eva Lang \footnotemark[1]\ \footnotemark[2]\ \footnotemark[3] \and Wilhelm Stannat \footnotemark[1]\ \footnotemark[2]}

\begin{document}

\maketitle
\renewcommand{\thefootnote}{\fnsymbol{footnote}}
\footnotetext[1]{Institut f\"ur Mathematik, Technische Universit\"at Berlin, D-10623 Berlin, Germany (lang@math.tu-berlin.de, stannat@math.tu-berlin.de)}
\footnotetext[2]{Bernstein Center for Computational Neuroscience, D-10115 Berlin, Germany. The work was supported by the BMBF, FKZ01GQ1001B.}
\footnotetext[3]{The work of this author was supported by the DFG RTG 1845.}
\renewcommand{\thefootnote}{\arabic{footnote}}

\begin{abstract}
\noindent
Stability of the traveling wave solution to a general class of one-dimensional nonlocal evolution equations is studied in $L^2$-spaces, thereby providing an alternative approach to the usual spectral analysis with respect to the supremum norm. We prove that the linearization around the traveling wave solution satisfies a Lyapunov-type stability condition in a weighted space $L^2(\rho)$ for a naturally associated density $\rho$. The result can be applied to obtain stability of the traveling wave solution under stochastic perturbations of additive or multiplicative type. For small wave speeds, we also prove an alternative Lyapunov-type stability condition in $L^2(\m)$, where $\m$ is the symmetrizing density for the traveling wave operator, which allows to derive a long-term stochastic stability result.
\end{abstract}

\section{Introduction}
Consider the nonlocal evolution equation 
\begin{equation}
\label{eq:evoleq}
\partial_t u(x,t) = d \partial_{xx} u(x,t) + S(u, w\ast g(u))(x,t),
\end{equation}
for $x\in \R$ and $t\geq 0$. Here, $d\geq 0$, $g\in C^1(\R)$ is strictly increasing, $S(x,g)\in C^1(\R\times\R)$ is strictly increasing in $g$, $w$ is a probability density that is differentiable almost everywhere with $\int |w_x|\, dx < \infty$ and $w\ast h (x) := \int w(x-y)h(y)\, dy$ denotes convolution. In order to ensure the existence of monotone traveling wave solutions, we suppose that $x\mapsto S(x, g(x))$ is bistable: there exist exactly three zeroes $a_1<a<a_2$ such that $S(a_i, g(a_i))=S(a,g(a))=0$, $\frac{d}{dx} S(a_i, g(a_i)) < 0$, $i=1,2$, and $\frac{d}{dx} S(a, g(a))>0$. 

\smallskip 
\noindent 
A strictly monotone traveling wave solution to \eqref{eq:evoleq} connecting the stable states is a solution of the form 
$$ 
u^{TW} (x,t) = \twp (x- ct) 
$$ 
for some wave profile $\twp\in C^1(\R)$ with $\twp_x >0$, $\lim_{x\rightarrow -\infty} \twp(x) = a_1$, $\lim_{x\rightarrow \infty} \twp(x) = a_2$, and some wave speed $c\in\R$. Inserting $u^{TW}$ into the equation \eqref{eq:evoleq} implies that $\twp$ satisfies the equation 
\begin{equation} 
\label{eq:twequation} 
-c\partial_x\twp = d\partial_{xx} \twp + S(\twp, w\ast g(\twp ))\, . 
\end{equation} 

\medskip 
\noindent 
The main example we have in mind is the 
\begin{itemize}
\item \emph{Neural Field Equation.}
\begin{equation}
\label{eq:nfe}
\partial_t u(x,t) = -u(x,t) + w\ast F(u(\cdot,t))(x),
\end{equation}
where $F$ is a sigmoid function. Equation \eqref{eq:nfe} has been introduced by S. Amari in \cite{amari} to study pattern formation in homogeneous recurrent neural networks and has since then been intensively studied in the (computational) neuroscience literature. Existence and uniqueness (up to spatial translation) of monotone traveling wave solutions was first proven in \cite{ermentroutmcleod}. 
\end{itemize}

\bigskip 
\noindent 
A first result on (exponential) stability of traveling wave solutions to nonlocal evolution equations of type \eqref{eq:evoleq} has been obtained in \cite{demasigobronpresutti} w.r.t. the sup-norm for the continuum limit of the 1-dimensional  
\begin{itemize}
\item \emph{Ising Model.} 
\[\partial_t u(t) = \tanh(\beta(w\ast u(t)+h))-u(t),\]
where $\beta >1$, $0\leq w\in \mathcal{C}^2$ is even and supported on $[-1,1]$, 
for small $h\ge 0$. The result has then been extended to general $h\ge 0$ in \cite{orlanditriolo}.
\end{itemize} 
In \cite{chenxinfu}, Chen proved existence, uniqueness, and exponential stability (again w.r.t. the sup-norm) of monotone traveling wave solutions to (\ref{eq:evoleq}) for a large class of evolution equations of the form (\ref{eq:evoleq}). Apart from the neural field equation and the Ising model, his results also cover the following examples: 
\begin{itemize}
\item \emph{Convolution Model for Phase Transitions.} 
\[\partial_t u(t) = \lambda w\ast u(t)-u(t)+f(u(t)),\]
where $\lambda>0$, $0\leq w \in \mathcal{C}^1$ is even 
and $f$ is bistable. Existence and Uniqueness of a monotone traveling front is established in \cite{batesfiferenwang}.
\item \emph{Thalamic Model.}
\[\partial_t u(t) = - u(t) + h (1-u(t)) F(w\ast(u^p(t))-\Theta),\]
where $h, \Theta>0$, $p \in {1,2,3,4}$, $w(x) = \frac{1}{2\sigma}e^{-\frac{|x|}{\sigma}}$ for some $\sigma>0$, and $F$ is a sigmoid function. 
\item The above example is included as a special case of nonlocal evolution equations of the form
\[\partial_t u(t) = r(u(t))+p(u(t))S(w\ast q(u(t)))\]
considered in \cite{chenermentroutmcleod}, where existence and uniqueness of monotone traveling waves is shown. We refer to \cite{chenermentroutmcleod} for the precise assumptions on the parameters.
\end{itemize}

\smallskip 
\noindent 
In all of the work cited above, stability of the traveling wave solution to (\ref{eq:evoleq}) is established in $L^{\infty}(\R)$ or $C_0(\R)$, the space of continuous functions vanishing at infinity. In \cite{bateschen}, Bates and Chen prove that the linear operator appearing in the equation when linearizing around the traveling wave solution has a spectral gap in $C_0$. More recently, reference \cite{zhaoruan} establishes spectral properties of traveling wave solutions to nonlocal evolution equations of type \eqref{eq:evoleq} in $L^p$, $1 \leq p \leq \infty$. 

\medskip 
\noindent
In this paper we will be concerned with the stability of traveling wave solutions in $L^2$-spaces. More precisely, considering equation \eqref{eq:evoleq} in moving frame coordinates $u^\# (x,t) =  u(x+ct,t)$, turning the traveling wave into a standing wave, yields the nonlocal evolution equation 
\begin{equation}
\label{eq:frozenwave}
\partial_t u^{\#}(x,t) = c \partial_x u^{\#}(x,t) + d\partial_{xx} u^{\#}(x,t) + S(u^{\#}, w\ast g(u^{\#}))(x,t).
\end{equation}
and $\twp$ (and its spatial translates) become stationary solutions to \eqref{eq:frozenwave}. Linearizing the right hand side around $\twp$ yields the frozen wave operator 
\[L^{\#} v = \partial_1 S(\twp, w\ast g(\twp)) v + c\partial_x v + d\partial_{xx} v +  \partial_2 S(\twp, w\ast g(\twp)) w\ast (g'(\twp)v).\]
Differentiating equation \eqref{eq:twequation} w.r.t.~$x$ yields that $L^{\#}\twp_x=0$, hence $0$ is 
an eigenvalue of the linear operator $L^\#$, and we say that $\twp_x$ is spectrally stable in $L^2$ if $\sup\{\mbox{Re} (\lambda ) : \lambda\in\mathbb C\setminus {\mathcal R} (L^\# )\} \setminus \{0\} < 0$, where 
${\mathcal R}(L^\# )$ denotes the resolvent set, i.e.~the set of all $\lambda\in\mathbb C$ for which the operator $\lambda - L^\#$ is invertible in $L^2$ with bounded inverse. It is well-known that this condition is implied by the stronger Lyapunov-type stability: there exist $\kappa, Z >0$ such that
\begin{equation}
\label{eq:spectralgapintuition}
\langle L^{\#} v, v\rangle \leq -\kappa \|v\|^2 + Z \langle v, \twp_x\rangle^2\, . 
\end{equation}
The geometric interpretation of \eqref{eq:spectralgapintuition} is that $L^2$-solutions of 
$\partial_t v (x,t) = L^\# v(x,t)$ will decay exponentially in directions orthogonal to the 
eigenspace generated by $\twp_x$ (and its spatial translates), whereas they will not decay in directions tangential to the traveling wave solutions. 

\smallskip 
\noindent 
The operator $L^\#$ is not symmetric w.r.t.~the inner product on the space $L^2$, so that spectral stability does not imply \eqref{eq:spectralgapintuition}. On the other hand, \eqref{eq:spectralgapintuition} is more robust and allows to study stability of \eqref{eq:evoleq} 
and stochastic perturbations with additive or multiplicative noise in the phase space with the direct Lyapunov method. Assumption \eqref{eq:spectralgapintuition} has been made in \cite{kruegerstannat} (in the neural field example) to study the behavior of the traveling wave solution under noise. In \cite{inglismaclaurin}, the stability and long-term behavior of traveling waves under noise are studied in a general setting under the assumption of spectral $L^2$-stability. However, it seems that condition \eqref{eq:spectralgapintuition} is difficult to verify in particular examples, in particular for large wave speeds $c$. 

\smallskip 
\noindent 
We will therefore study \eqref{eq:spectralgapintuition} w.r.t.~a different measure $\rho (x)\, dx = \frac{\psi}{\twp_x}(x)\, dx$, where $\psi$ is the eigenfunction of the adjoint operator $L^{\#,*}$ 
corresponding to the eigenvalue $0$. It can be shown under mild assumptions that $\psi$ exists and is strictly positive. Then 
\[\langle L^{\#}v, \twp_x \rangle_{\rho} =  \langle L^{\#}v, \psi \rangle = \langle v, L^{\#,*}\psi \rangle = 0\] 
for every $v \in H^2(\rho)$, which shows that components of the linearized frozen wave equation 
pointing towards the direction of the traveling wave solutions and its orthogonal components are infinitesimally separated in $L^2 (\rho )$. Our main result Theorem \ref{thm:spectralgap} then proves the $L^2 (\rho)$-version 
\begin{equation}
\label{eq:sgfrozenwaveintuition}
\langle L^{\#}v, v\rangle_{\rho} \leq -\kappa \|v\|_{\rho}^2 + Z \langle v, \twp_x \rangle_{\rho}^2
\end{equation}
of (\ref{eq:spectralgapintuition}) under rather general assumptions. 
(\ref{eq:sgfrozenwaveintuition}) allows to derive stability results up to a finite time horizon $T$. 
In \cite{langMSA} it was shown for the neural field example how (\ref{eq:sgfrozenwaveintuition}) can be used to describe the influence of the noise on multiple scales, and in particular to express the (stochastic) stability of the wave. 

\smallskip 
\noindent 
The analogue of $\rho$ in the local case, i.e.~the case, where $\twp$ is the traveling wave solution 
$-c\twp_x = d\twp_{xx} + f(\twp )$ to the reaction-diffusion equation 
$\partial_t v = d\partial_{xx} v + f(v)$ with bistable reaction term $f$, is given by $e^{\frac cd x}$, since in this case $\psi (x) = e^{\frac cd x}\twp_x $ is an eigenfunction of the adjoint operator $L^{\#,*} v = - c\partial_x v + d \partial_{xx} v + f^\prime (\twp ) v$. Writing  $v = e^{-\frac c{2d} x}h$ turns the frozen wave operator into a Schr\"odinger operator, since $L^\# v = e^{-\frac c{2d} x}\left( dh_{xx} + (f^\prime (\twp ) - \frac{c^2}{2d}) h\right)$, and turns condition
\eqref{eq:sgfrozenwaveintuition} in this case into 
$$ 
\int (d h_{xx} +  (f^\prime (\twp ) - \frac{c^2}{4d})) h\, h\, dx \le - \kappa \int h^2\, dx + Z\langle h, e^{\frac c{2d} x}\twp_x\rangle^2 
$$ 
which is used in \cite{fifemcleod} to obtain the stability of traveling wave solutions w.r.t.~the sup-norm. 

\smallskip 
\noindent 
Let us return to the nonlocal case. The drawback of condition \eqref{eq:sgfrozenwaveintuition} is that 
on larger time scales, we lose control in $L^2 (\rho )$ over the nonlinear part of the dynamics. 
Here it would be more suitable to work in $L^2$ (w.r.t.~the Lebesgue measure). If $c=0$, the $L^2(\rho)$-norm is equivalent to the $L^2$-norm, but this is typically not the case if $c \neq 0$. 
In the second part of this article, we therefore study in particular the long-term stability of the traveling wave. We extend the $L^2$-stability for the case $c=0$ to small wave speeds $c$ by a perturbation argument and show how it can be used to derive a long-term stochastic stability result. 

\smallskip 
\noindent 
The article is structured as follows.
In section \ref{subsection:sgsetting} we describe the mathematical setting.
Our main theorem, the spectral gap inequality in $L^2(\rho)$, is stated and proved in section \ref{subsection:sgineqrho}.
We show that the assumptions we make are satisfied in a very general setting, in particular the result applies to the examples stated above (section \ref{subsection:examples}).

\smallskip 
\noindent 
The long-term $L^2$-stability is studied in section \ref{section:L2stab}. We carry out the perturbation argument (section \ref{subsection:sgsmallc}) and obtain a second version of the spectral gap inequality for small wave speeds $c$. We show how the smallness condition on $c$ can be translated into a condition on the parameters of the system by deriving bounds on the wave speed in the example of the neural field equation when the strength of the synaptic connections is modeled by a two-sided exponential kernel (section \ref{subsection:boundsc}). 
Finally, we derive a long-term stochastic stability result for the neural field example (section \ref{subsection:longtermstab}).

\section{$L^2(\rho)$-Stability}

\subsection{The Setting}
\label{subsection:sgsetting}
We denote by $H^k = W^{k,2}$ the Sobolev space of $k$-times weakly differentiable functions equipped with the inner product
\[ \langle u, v\rangle_{H^k} = \sum_{i=0}^k \int u^{(i)}(x) v^{(i)}(x)\, dx .\]
Analogously, for a continuous density $ \mu: \R \rightarrow (0,\infty)$, we denote by $H^k(\mu)$ the weighted Sobolev space with inner product
\[ \langle u, v \rangle_{H^k(\mu)} = \sum_{i=0}^k \int u^{(i)}(x) v^{(i)}(x) \mu(x) dx\]
and set $L^2(\mu)= H^0(\mu)$.

Motivated by the traveling wave examples given above we consider the operator
\[L^{\#}v = -fv+c\partial_x v+d\partial_{xx}v + r w\ast(qv), \qquad D(L^{\#}) = H^2\]
(if $d=0$, then $D(L^{\#})=H^1$).
We make the following assumptions on the parameters. 
\begin{itemize}
	\item $c\in\R$, $d \geq 0$
	\item $f,r,q \in C(\R)$, $r>0$ and $q>0$
	\item $f, r$, and $q$ are bounded, and both, $\inf_{x\in\mathbb R} q(x)$ and 
	      $\inf_{x\in\mathbb R} r(x)$, are strictly positive 
	\item $w\geq 0$ is differentiable almost everywhere, $\int w(x) dx = 1$, and $\int |w_x(x)| dx <\infty$ 
\end{itemize}
Note that in the case of the frozen wave operator $L^{\#}$ associated to (\ref{eq:evoleq}) 
we have that 
\[
f = - \partial_1 S(\twp, w\ast g(\twp)), \quad r = \partial_2 S(\twp, w\ast g(\twp)), 
\quad q = g'(\twp).
\]

\noindent 
We decompose $L^{\#}$ into a local and a nonlocal part,
\[L^{\#}v = Av + Pv,\]
where the local part is given by
\[Av =  -f v   + c\partial_x v + d \partial_{xx}v,\]
and the nonlocal part is
\[Pv = \int p(x,y) v(y) dy\]
with
\[p(x,y) = r(x) w(x-y) q(y).\]
The adjoint of $L^{\#}$ is
\[L^{\#,*}v = A^*v + P^*v, \qquad D(L^{\#,*})=H^2,\]
where the local part is
\[A^*v =  - f v  - c\partial_x v + d \partial_{xx}v,\]
and the nonlocal part is
\[P^*v = \int p^*(x,y) v(y) dy\]
with
\[p^*(x,y) = p(y,x).\]

\begin{assumption*}
There exists a unique (up to constant multiples) $0 \not\equiv \twp_x \in H^2$ such that $L^{\#}\twp_x=0$ and a unique (up to constant multiples) $0\not\equiv \psi \in H^2$ such that $L^{\#,*}\psi = 0$, and $\twp_x>0$ and $\psi> 0$.
\end{assumption*}
Here we denote the eigenfunction of $L^{\#}$ by $\twp_x$ in reference to the traveling wave example.
Concerning existence of the adjoint eigenfunction $\psi$ we note the following.
\begin{proposition}
Assume that $f\geq 0$, that 
\begin{equation}
\label{eq:bistab}
\lim_{x\rightarrow \pm \infty} -f(x)+ r(x) q(x) < 0, 
\end{equation}
corresponding to the bistability of $S$ in (\ref{eq:evoleq}), and that there exists a unique $0 \not\equiv \twp_x \in H^2$ such that $L^{\#}\twp_x=0$.
Then there exists a unique $\psi \in H^2$ such that $L^{\#,*}\psi = 0$, and $\psi> 0$.
\end{proposition}
\begin{proof}
Similar to the proof in the neural field setting, cf.~Prop.~2.2 in \cite{langMSA}, we can decompose 
\[L^{\#, *} = K + B,\]
where for some $M>0$,
\[Kv(x) = \mathds{1}_{[-M, M]}(x) P^*v (x) + \mathds{1}_{[-M, M]^c}(x)  \int_{-M}^M p^*(x,y) v(y)\, dy\]
and 
\[Bv = -fv - c\partial_x v + d \partial_{xx} v + \mathds{1}_{[-M, M]^c}(x) \int_{[-M, M]^c} p^*(x,y) v(y)\, dy.\]
The operator $K$ is Hilbert-Schmidt and hence compact.
Using (\ref{eq:bistab}) and the fact that $f\geq 0$, it can be shown that if we choose $M$ large 
enough, there exists $\delta>0$ such that $\langle Bv, v\rangle \leq -\delta \|v\|^2$, so that 
$B$ has a bounded inverse. It follows that $B^{-1}L^{\#,*} = I + B^{-1}K$, and $B^{-1}K$ is 
compact. Therefore there exists a unique $\psi\in H^2$ such that $L^{\#,*}\psi = 0$. 

We show that $\psi$ is of one sign. Assume without loss of generality that there exists $x$ s.t.~$\psi(x)>0$ and set $\psi^+ = \psi \vee 0$. Note that $L^{\#,*}$ generates a $C_0$-semigroup in $L^2$. Define $R_{\alpha} = (\alpha-L^{\#,*})^{-1}$
to be the associated resolvent. Note that for $\alpha$ large enough, $R_{\alpha}$ is positivity preserving, that is, $u\geq 0$ implies that $R_{\alpha}u\geq 0$.
It follows that 
\[\alpha R_{\alpha} (\psi^+) = \alpha ( R_{\alpha} \psi +  R_{\alpha} \psi^-) \geq \alpha R_{\alpha} \psi \vee 0 = \psi^+ \, , 
\]
hence $\alpha R_\alpha (\psi^+ ) - \psi^+ \ge 0$. Then, since $\twp_x>0$, 
\[
\langle \twp_x , \alpha R_\alpha (\psi^+)-\psi^+\rangle = 
\langle \twp_x , L^{\#,*} R_\alpha (\psi^+) \rangle = \langle L^{\#}\twp_x, R_\alpha (\psi^+ ) \rangle = 0 
\]
and thus $\alpha R_\alpha (\psi^+ ) = \psi^+$, hence $\psi^+\in D(L^{\#,*})$ and 
$L^{\#, *} \psi^+ = 0$, hence $\psi = \psi^+$. To see that $\psi$ is strictly positive, set
\[dX_t = \sqrt{2d}\, dW_t - c\, dt\]
and note that for $m := \|f\|_{\infty}$,
\begin{align*}
d(e^{-mt}\psi(X_t)) 
& = e^{-mt}\psi_x(X_t)\, dW_t + e^{-mt} (d\partial_{xx}-c\partial_x - m) \psi(X_t)\, dt\\
& \leq e^{-mt}\psi_x(X_t)\, dW_t + e^{-mt} L^{\#,*}\psi(X_t)\, dt = e^{-mt}\psi_x(X_t)\, dW_t,
\end{align*}
and thus for all $x\in\R$,
\[\psi(x) \geq E_x(e^{-mt}\psi(X_t)) > 0.\]
\end{proof}

We normalize $\psi$ such that $\langle \twp_x, \psi \rangle = 1$, so that 
\begin{equation}
\label{eq:nu}
\nu(dx) = \twp_x \psi\, dx
\end{equation}
is a probability measure, and introduce the density 
\[\rho(x) = \frac{\psi(x)}{\twp_x(x)}.\]
We assume that there exists $K_{\rho}$ such that
\begin{equation}
\label{eq:rho2}
 w\ast \rho(x) \leq K_{\rho} \rho(x).
\end{equation}
This implies that the frozen wave operator $L^{\#}:H^2(\rho)\rightarrow L^2(\rho)$ is well-defined.

\subsubsection{Reformulation}
\label{subsubsection:reformulation}

We want to prove that there exists $\kappa>0$ such that for all $v\in H^2(\rho)$
\begin{equation}
\label{eq:spectralgapintro}
\langle L^{\#} v, v\rangle_{\rho} \leq -\kappa \Big( \|v\|_{\rho}^2 - \langle v, \twp_x\rangle_{\rho}^2\Big).
\end{equation}
The following representation of the energy $\langle A u,u\rangle_\rho$ related to the local part 
as a sum of squares will be useful: suppose that $v = h\twp_x$ for some function $h\in C^2_c (\mathbb R)$. 
Then 
$$ 
\begin{aligned} 
A\left( h\twp_x\right) 
& = hA\twp_x +  \twp_x \left( d\partial_{xx}h + \left( c + 2d\frac{\partial_x \twp_x}{\twp_x}\right) \partial_x h \right) \\ 
& = - hP\twp_x + \twp_x\left( d\partial_{xx}h + \left( c + 2d\frac{\partial_x \twp_x}{\twp_x}\right) \partial_x h \right) \, . 
\end{aligned}  
$$ 
Hence, integration against $h\twp_x\, \rho \, dx$ and integration by parts yield 
\begin{equation}
\label{Stannat1} 
\begin{aligned} 
\int A& \left( h\twp_x\right) h\twp_x\, \rho  \, dx 
= - \int h^2 P\twp_x \psi\, dx - d\int \left( \partial_x h\right)^2 \twp_x \psi\, dx \\ 
& \hspace{3.5cm} + \int \left( c + d\frac{\partial_x \twp_x}{\twp_x} - d\frac{\partial_x \psi}{\psi} \right) 
\partial_x h h\twp_x\psi\, dx \\ 
& = - \int h^2 P\twp_x \psi\, dx - d\int \left( \partial_x h\right)^2 \twp_x \psi\, dx \\ 
& \hspace{3.5cm}  - \frac 12 \int \partial_x \left(\twp_x \psi \left( c + d\frac{\partial_x \twp_x}{\twp_x}  
  - d\frac{\partial_x \psi}{\psi} \right)\right) h^2 \, dx \\ 
& = - \frac 12 \int h^2 P\twp_x \psi\, dx - d\int \left( \partial_x h\right)^2 \twp_x \psi\, dx   - \frac 12 \int h^2\twp_x P^\ast\psi\, dx 
\end{aligned} 
\end{equation} 
thereby using the identity 
$$ 
\begin{aligned} 
\partial_x  \left(\twp_x \psi \left( c + d\frac{\partial_x \twp_x}{\twp_x} - d\frac{\partial_x \psi}{\psi} \right)\right) 
& = c\partial_{x}\left( \twp_x\psi\right) + d\partial_x \left( \partial_x \twp_x \psi - \twp_x \partial_x\psi\right) \\ 
& = \left( c\partial_{x}\twp_x + d\partial_{xx}\twp_x \right) \psi + \left( c\partial_x\psi - d\partial_{xx}\psi\right)\twp_x \\ 
& = A\twp_x \psi - \twp_x A^* \psi = \twp_x P^* \psi - P\twp_x\, \psi\, . 
\end{aligned} 
$$
To reformulate the nonlocal part of the energy in a similar manner, it is convenient to introduce the integral operator 
$$
P_0 h := \frac{P(h\twp_x)}{P\twp_x}  \qquad\mbox{ for } v = h\twp_x \, , h\in L^2 (\nu )\, . 
$$ 
Note that $P_0\mathds{1} \equiv \mathds{1}$, so that 
\[p_0(x,y) = \frac{p(x,y)\twp_x(y)}{P\twp_x(x)} = \frac{w(x-y) q(y) \twp_x(y)}{w\ast(q\twp_x)(x)}\]
is a Markov kernel. Moreover, let us define the probability measures 
\begin{equation}
\label{eq:mu}
\mu(x) = \frac{1}{Z_{\mu}} P\twp_x(x)\psi(x), \hspace{.5cm} \mu^*(x) = \frac{1}{Z_{\mu}} \twp_x(x) P^*\psi(x),
\end{equation}
where $Z_{\mu}=\int P\twp_x(x) \psi(x) dx = \int \twp_x(x) P^*\psi(x) dx$ is a normalizing constant. 
With these notations we can reformulate the nonlocal part as 
\begin{equation}
\label{Stannat2} 
\begin{aligned} 
\int P(h\twp_x) h\twp_x\, \rho\, dx 
& = Z_\mu \int P_0 h \, h\, d\mu = Z_\mu\mbox{Cov}_\mu (P_0 h,h) + Z_\mu \int P_0h \, d\mu \int h\, d\mu 
\end{aligned} 
\end{equation} 
so that we can combine \eqref{Stannat1} and \eqref{Stannat2} to obtain 
\begin{equation}
\label{Stannat3} 
\begin{aligned} 
\int L^\#\left( h\twp_x\right) h\twp_x \rho\, dx 
& =  - \frac{1}{2} \int h^2 P\twp_x \psi\, dx - d\int \left( \partial_x h\right)^2 \twp_x \psi\, dx 
 - \frac 12 \int h^2\twp_x P^\ast\psi\, dx \\ 
& \qquad + Z_\mu\mbox{Cov}_\mu (P_0 h,h) + Z_\mu \int P_0h \, d\mu \int h\, d\mu \\ 
& =  - \frac{Z_\mu}2 \mbox{Var}_\mu (h) - \frac{Z_\mu}2 \mbox{Var}_{\mu^\ast} (h) 
     - \frac{Z_\mu}2 \left( \int h\, d\mu - \int h\, d\mu^\ast\right)^2 \\
& \qquad  + Z_\mu\mbox{Cov}_\mu (P_0 h,h)  - d\int \left( \partial_x h\right)^2 d\nu  \\ 
& \le\frac{Z_\mu}2 \mbox{Var}_\mu (P_0 h) - \frac{Z_\mu}2 \mbox{Var}_{\mu^\ast} (h) 
  - \frac{Z_\mu}2 \left( \int h\, d\mu - \int h\, d\mu^\ast\right)^2 \\ 
& \qquad  - d\int \left( \partial_x h\right)^2 d\nu \, . 
\end{aligned} 
\end{equation} 
The last inequality then extends to all $v = h\twp_x\in H^2 (\rho )$ using a simple approximation.

\subsection{Spectral Gap Inequality in $L^2(\rho)$}
\label{subsection:sgineqrho}

In order to prove (\ref{eq:spectralgapintro}), by (\ref{Stannat3}), we need to estimate $\mbox{Var}_\mu (P_0 h)$ against $\mbox{Var}_{\mu^\ast} (h)$. We will use the following result on the $L^2$-contractivity of Markovian integral operators, which is of independent interest. 

\begin{lemma} 
\label{lem:L2contractivity} 
Let $\nu$ be a probability measure on $\mathbb R$ with strictly positive continuous density, 
$k : \mathbb R^2 \to [0, \infty )$ be measurable such that $\int k(x,y)\, dy = 1$ for all $x$. Assume that 
\begin{itemize} 
\item[(i)] $k$ is differentiable almost everywhere w.r.t.~$x$ and  
$$ 
M := \text{ess-sup}_{x\in\mathbb R} \int \frac{k_x (x,y)^2}{k(x,y)}\, dy < \infty
$$
\item[(ii)] $\exists$ $\kappa_0 < \infty$ such that 
$$ 
\mbox{Var}_\nu (h) \le \kappa_0 \int h_x^2 \, d\nu \qquad\forall h\in H^1 (\nu )\,  . 
$$
\end{itemize}   
Denote by $Kh(x) = \int k(x,y)h(y)\, dy$ the Markovian integral operator associated with $k$. 
Then 
$$ 
\mbox{Var}_\nu (Kh) \le \frac{\kappa_0 M}{1 + \kappa_0 M} \mbox{Var}_{\nu K} (h) 
\qquad\forall h\in L^2 (\nu )\,  . 
$$
Here $\nu K$ is the probability measure on $\mathbb R$ defined by 
$\int h\, d\nu K = \int Kh\, d\nu$. 
\end{lemma} 

\begin{proof} 
First assume that $h\in{\mathcal B}_b ( \mathbb R)$. Then $Kh (x) = \int k(x,y)h(y)\, dy$ 
is differentiable almost everywhere with 
\begin{align*}
\left( \partial_x Kh(x)\right)^2 
& = \left( \int \partial_x k(x,y) h(y) \, dy\right)^2 
 \le \int \frac{\partial_x k(x,y)^2}{k(x,y)} \, dy \int k(x,y) h^2 (y) \, dy 
\le M K h^2 (x) \, . 
\end{align*}
In particular, $Kh\in H^1 (\nu )$. Moreover, $K\mathds{1} = \mathds{1}$ implies that 
$$ 
0 = \partial_x K\mathds{1}(x) = \int \partial_x k(x,y)\, dy \, , 
$$ 
hence 
\begin{align*}
\left( \partial_x Kh(x)\right)^2 
& = \left( \int \partial_x k(x,y) \left( h(y) - Kh (x)\right) \, dy\right)^2 \\ 
& \le \int \frac{\partial_x k(x,y)^2}{k(x,y)} \, dy  
\int k(x,y) \left( h - Kh (x)\right)^2(y) \, dy \\ 
& \le M \left( K\left( h^2\right) - \left( Kh (x) \right)^2\right)\, . 
\end{align*}
It follows that 
\begin{align*}
\mbox{Var}_\nu (Kh) 
& \le \kappa_0 \int \left( \partial_x Kh\right)^2 \, d\nu 
\le \kappa_0 M \left( \int K\left( h^2\right)\, d\nu - \int \left( Kh\right)^2\, d\nu\right)
 \,  , 
\end{align*}
hence 
\begin{align*}
(1 + \kappa_0 M)\mbox{Var}_\nu (Kh) 
& \le \kappa_0 M \left( \int K\left( h^2\right)\, d\nu  - \left( \int Kh\, d\nu\right)^2 \right)\\ 
& =  \kappa_0 M \left( \int  h^2\, d\nu K - \left( \int h\, d\nu K \right)^2\right)  
\end{align*}
which implies the assertion for bounded $h$. The general case then follows by approximation. 
\end{proof}

Denote by $\mathcal{S}$ the support of $w$, $\mathcal{S} = \{ x\in \R: w(x) > 0 \}$.
We make the following additional assumption on $w$.
\begin{assumption}
\label{ass:w}
\leavevmode
\begin{enumerate}[(i)]
	\item for all $v\in L^2(\rho)$, $\partial_x w\ast v = w_x \ast v$
	\item $M := \sup_{x\in\R} \int_{x-\mathcal{S}} \big( \frac{w_x(x-y)}{w(x-y)}\big)^2 p_0(x,y) dy <\infty$
\end{enumerate}
\end{assumption}
Let $\nu, \mu, $ and $\mu^*$ be as in (\ref{eq:nu}) and (\ref{eq:mu}).
\begin{theorem}
\label{thm:spectralgap}
Assume that (\ref{eq:rho2}) and Assumption \ref{ass:w} are satisfied and that furthermore
\begin{enumerate}[(i)]
\item there exist $\delta_i, \delta^\ast_i > 0, i=1,2,$ such that
\[
\delta_1 \twp_x  \leq P\twp_x \leq \delta_2 \twp_x, \qquad
\delta^\ast_1 \psi \leq P^*\psi \leq \delta^\ast_2 \psi.\]
In particular, the $\nu$-, $\mu$-, and $\mu^*$-norms are equivalent.
\item there exists $\kappa_0 > 0$ such that for all $h\in H^1(\mu)$,
\begin{equation}
\label{eq:poincare}
Var_{\mu}(h) \leq \kappa_0 \int h_x^2(x) \mu(dx)
\end{equation}
\end{enumerate}
Then for all $v\in H^2(\rho)$,
\[\langle L^{\#}v ,v \rangle_{\rho} \leq -\kappa \Big( \|v\|_{\rho}^2 - \langle v, \twp_x\rangle_{\rho}^2\Big),\]
where 
\[\kappa = \frac{\delta^\ast_1}{2} \Big(1-\frac{\kappa_0M}{1+\kappa_0M}\Big).\]
\end{theorem}

\begin{proof} 
\eqref{Stannat3} implies that 
\begin{equation*}
\begin{aligned} 
\int L^\#\left( h\twp_x\right) h\twp_x \rho\, d\mu 
& \le\frac{Z_\mu}2 \mbox{Var}_\mu (P_0 h) - \frac{Z_\mu}2 \mbox{Var}_{\mu^\ast} (h) \, . 
\end{aligned}
\end{equation*} 
Applying Lemma \ref{lem:L2contractivity} to the measure $\mu$ and the kernel $p_0$, using that $\mu^* = \mu P_0$, we obtain that for $\gamma = \frac{\kappa_0 M }{1+\kappa_0 M} < 1$,  
$$ 
\mbox{Var}_\mu (P_0 h) \le \gamma  \mbox{Var}_{\mu^\ast} (h) \quad\forall h\in L^2 (\mu^\ast )\, .  
$$
\smallskip 
\noindent 
Combining both estimates we arrive at 
\begin{align*}
	\langle L^{\#}v,v\rangle_{\rho} \leq - \frac{(1-\gamma)Z_{\mu}}{2} Var_{\mu^*}(h),
\end{align*}
and since (i) implies 
\[Z_{\mu} Var_{\mu^*}(h) = Z_{\mu} \int \big(h(x)-E_{\mu^*}(h)\big)^2 \mu^*(dx) \geq \delta^\ast_1 \int \big(h(x)-E_{\mu^*}(h)\big)^2 \nu(dx) \geq\delta^\ast_1 Var_{\nu}(h),\]
we conclude that 
\[\langle L^{\#}v,v\rangle_{\rho} \leq -\kappa \big(\|v\|_{\rho}^2 - \langle v, \twp_x\rangle^2_{\rho}\big)\]
with $\kappa = \frac{\delta^\ast_1(1-\gamma)}{2}$.
\end{proof}

\subsection{Application to the Examples}
\label{subsection:examples}

We show that the assumptions in Theorem \ref{thm:spectralgap} are satisfied under rather general conditions. 

\begin{remark}
\label{rem:muckenhoupt}
Using a result by Muckenhoupt on Hardy's inequalities with weights (originally obtained by Tomaselli, Talenti, Artola, cf. \cite{muckenhoupt}, Thm. 1),
assumption (ii) in Theorem \ref{thm:spectralgap} is satisfied if and only if
\[B_1:=\sup_{r>0} \int_r^{\infty} \mu(x) dx \int_0^r \frac{1}{\mu(x)}dx < \infty\]
and
\[B_2 := \sup_{r>0} \int_{-\infty}^{-r} \mu(x) dx \int_{-r}^0 \frac{1}{\mu(x)} dx < \infty.\]
In this case we can bound $\kappa_0$ in (\ref{eq:poincare}) by
\[ B_1 \wedge B_2 \leq \kappa_0 \leq 4(B_1\vee B_2). \]
\end{remark}

\begin{theorem}
\label{thm:application}
Assume that $w>0$ in a neighborhood of $0$ and that (\ref{eq:rho2}) and Assumption \ref{ass:w} are satisfied. Assume further that
there exist $\alpha,\beta, k, l>0$, such that for all $x\geq 0$, $y\geq 0$
\begin{equation}
\label{eq:expdecay}
\mu(x+y)   \leq k e^{-\alpha y} \mu(x),	\qquad \mu(-x-y)   \leq l e^{-\beta y} \mu(-x),
\end{equation}
and that
\begin{equation}
\label{eq:boundderivative}
\Big\|\frac{\twp_{xxx}}{\twp_x}\Big\|_{\infty} + \Big\|\frac{\twp_{xx}}{\twp_x}\Big\|_{\infty}+ \Big\|\frac{\psi_{xx}}{\psi}\Big\|_{\infty} + \Big\|\frac{\psi_{x}}{\psi}\Big\|_{\infty} < \infty.
\end{equation}
Then the assumptions of Theorem \ref{thm:spectralgap} are satisfied.
\end{theorem}
\begin{proof}
(I) 
Since $m:= \big\|\frac{\twp_{xx}}{\twp_x}\big\|_{\infty}<\infty$ by (\ref{eq:boundderivative}), $-m\twp_x \leq \twp_{xx} \leq m\twp_x$ and hence for $x, y\geq 0$, $\twp_x(x+y) \geq e^{-my} \twp_x(x)$
and $\twp_x(-x-y) \geq e^{-my} \twp_x(-x)$.
It follows that for $x\geq 0$,
\begin{align*}
	P\twp_x(x) 
	& \geq \inf r \inf q \int_{-\infty}^0 w(y) \twp_x(x-y)dy\\
	& \geq \inf r \inf q \int_{-\infty}^0 w(y) e^{m y}  dy \ \twp_x(x),
\end{align*}
and analogously for $x\leq 0$. Thus, there exists $\delta_1>0$ such that
\[\delta_1 \twp_x(x) \leq P\twp_x(x).\]
Using (\ref{eq:boundderivative}) it follows that there exists $\delta_2>0$ such that
\[P\twp_x = -A\twp_x = f \twp_x -c\twp_{xx}-d\twp_{xxx} \leq \delta_2 \twp_x.\]
It can be proven analogously that there exist $\delta^\ast_1, \delta^\ast_2>0$ such that
\[
\delta^\ast_1 \psi \leq P^*\psi \leq \delta^\ast_2 \psi\, .
\]
Assumption (i) of Theorem \ref{thm:spectralgap} is therefore satisfied.

(II) 
\begin{align*}
	B_1 
	& := \sup_{r>0} \int_r^{\infty} \mu(x) dx \int_0^r \frac{1}{\mu(x)}dx \\
	& \leq  k^2 \int_r^{\infty} e^{-\alpha(x-r)}dx \mu(r)  \int_0^r e^{-\alpha (r-x)} dx \frac{1}{\mu(r)} \leq  \frac{k^2}{\alpha^2},
\end{align*}
and analogously
\[B_2 := \sup_{r>0} \int_{-\infty}^{-r} \mu(x) dx \int_{-r}^0 \frac{1}{\mu(x)} dx \leq  \frac{l^2}{\beta^2}.\]
Using Remark \ref{rem:muckenhoupt}, assumption (ii) of Theorem \ref{thm:spectralgap} is satisfied.
\end{proof}

\begin{remark}\leavevmode
\begin{enumerate}
	\item 
	It was proven in \cite{langMSA} that in the case of the neural field equation with $w(x) = \frac{1}{2\sigma}e^{-\frac{|x|}{\sigma}}$, $\sigma >0$, $\twp_x$ and $\psi$ decay exponentially, and that $\rho$ grows exponentially at a rate smaller than $\frac{1}{\sigma}$. Since $\big\| \frac{w_x}{w}\big\|_{\infty}<\infty$, it follows that in this case (\ref{eq:rho2}) and Assumption \ref{ass:w}, as well as (\ref{eq:expdecay}) and (\ref{eq:boundderivative}) are satisfied.
	\item In \cite{zhaoruan} it is shown in a rather general setting that for $q\equiv 1$ and $w$ satisfying 
	\[\int w(x) e^{\alpha x} dx < \infty\] for all $\alpha \in \R$, $\twp_x$ decays exponentially and the exact rates are given. Existence and exponential decay of the adjoint eigenfunction are also proven. In particular, (\ref{eq:expdecay}) and (\ref{eq:boundderivative}) are satisfied.
	\item If $\twp_x$ and $\psi$ decay exponentially, then $\rho$ (or $\frac{1}{\rho}$, depending on whether $c>0$ or $c<0$) grows exponentially. Thus, if $w$ has compact support and $\sup_{x\in \mathcal{S}} \big|\frac{w_x(x)}{w(x)}\big|<\infty$, or if $w$ decays faster than exponentially, then (\ref{eq:rho2}) and Assumption \ref{ass:w} are satisfied.
\end{enumerate}
\end{remark}

\section{Long-term $L^2$-Stability}
\label{section:L2stab}

In this section we will assume that $d=0$ and that $f\geq\inf f > 0$.
Another measure that is naturally associated with the problem is the symmetrizing measure of the traveling wave operator
\[L v = -fv +  r w\ast (q v)\]
with density
\[\m(x) = \frac{q(x)}{r(x)}.\]
Note that the $L^2(\m)$-norm is equivalent to the $L^2$-norm.

\subsection{Spectral Gap Inequality in $L^2(\m)$ for Small Wave Speeds}
\label{subsection:sgsmallc}

If $c=0$, then $\psi = \frac{1}{Z} \frac{q}{r} \twp_x$, where $Z  = \int \frac{q}{r} \twp_x^2 dx$, and thus  $\m= Z \rho$.
In this case, if the assumptions in Theorem \ref{thm:spectralgap} are satisfied, then $L$ has a spectral gap in $L^2(\m)$.

We can extend the spectral gap for the case $c=0$ to small wave speeds $c$ by a perturbation argument.

\begin{theorem}
\label{thm:spectralgapsmallc}
Assume that Assumption \ref{ass:w} is satisfied and that furthermore
there exists $\kappa_0^0>0$ such that for all $h\in H^1(\mu^0)$,
	\[Var_{\mu^0}(h) \leq \kappa_0^0 \int h_x^2(x) \mu^0(dx),\]
	where $\mu^0 = \frac{1}{Z_{\mu^0}} \frac{q}{r}P\twp_x \twp_x$ with $Z_{\mu^0} = \int \frac{q}{r}P\twp_x\twp_x dx$. 
Then there exists $c^* = c^*(w,f,r,q)>0$ (see (\ref{eq:c*}) for the precise definition) such that if $c=c(w,f,r,q)$ satisfies $|c| \leq c^*$, there exist $\kappa, Z>0$ such that
\begin{equation}
\label{eq:spectralgapsmallc}
\langle L v, v\rangle_{\m} \leq -\kappa \|v\|_{\m}^2 + Z \langle v, \twp_x\rangle_{\m}^2.
\end{equation}
\end{theorem}
\begin{proof}
Set $\varphi^0 = \frac{P\twp_x}{f}$ and $P^0v = \int p^0(x,y) v(y) dy$ where $p^0(x,y) = p(x,y) \frac{\twp_x(y)}{\varphi^0(y)}$.
Then 
$$
P^0\varphi^0 (x) = \int p^0 (x,y) \varphi^0 (y)\, dy  = \int p (x,y) \twp_x (y) \, dy  
= P\twp_x (x) = f\varphi^0 (x)\, ,   
$$ 
so that $L^0\varphi^0=0$, where
\[ 
L^0 v = -fv + P^0v, \qquad D(L^0)=L^2\, . 
\]

The $L^2$-adjoint operator is given by 
$L^{0,\ast} v = - fv + \frac{\twp_x}{\varphi^0} P^\ast v$ with eigenfunction 
$\psi^0 = \frac{1}{Z^0} \frac{q}{r}\twp_x$ corresponding to the eigenvalue $0$. Here, 
$$ 
Z^0=\int \frac{q}{r}\twp_x\varphi^0 dx = \int \frac{q}{rf} P\twp_x\twp_x dx \, . 
$$
We want to apply Theorem \ref{thm:spectralgap} with $L^0$ replacing $L^\#$ and $\varphi^0$ 
replacing $\twp_x$. In particular, $\rho^0 = \frac{\psi^0}{\varphi^0}$, so that 
$$  
\int L^0 (h\varphi^0) h\varphi^0 \rho^0 \, dx 
\le \frac{Z_\mu^0}2 \mbox{Var}_{\mu^0} (P_0^0 h) - \frac{Z_\mu^0}2 \mbox{Var}_{\mu^{0, \ast}} (h)   
\, , 
$$ 
with $P_0^0 h = \frac{P_0 (h\varphi^0 )}{\varphi^0} = \frac{P(h\twp_x)}{P\twp_x}$, which 
coincides with $P_0 h$ of the previous section. 

\medskip 
\noindent
Note that in this case $P^0 \varphi^0 = f\varphi^0$ and $P^{0,\ast}\psi^0 = f\psi^0$, so that 
assumption (i) of Theorem \ref{thm:spectralgap} is trivially satisfied with $\delta_1 = \inf f$   
and $\delta_2 = \sup f$. It follows that there exists $\kappa_0 > 0$ such that 
\begin{equation}
\label{eq:spectralgap0}
\langle L^0v, v\rangle_{\rho^0} \leq - \kappa_0 \big(\|v\|_{\rho^0}^2 - \langle v, \varphi^0\rangle_{\rho^0}^2 \big).
\end{equation}
Since 
$$
\m = \frac qr = Z_0\frac{\psi^0}{\twp_x} = Z_0\frac{\varphi^0}{\twp_x}\rho^0 
$$ 
and therefore 
\begin{align*}
\langle L v, v\rangle_{\m}
& = \langle -f v + Pv, v \rangle_{\m} = \Big\langle -f v + P^0 \Big(\frac{\varphi^0}{\twp_x}v\Big) , v \Big\rangle_{\m} \\
& = Z_0 \Big\langle -f \frac{\varphi^0}{\twp_x} v + P^0\Big(\frac{\varphi^0}{\twp_x} v\Big), v\frac{\varphi^0}{\twp_x}\Big\rangle_{\rho^0} - \Big\langle f \Big( 1-\frac{\varphi^0}{\twp_x}\Big) v, v\Big\rangle_{\m}\\
& = Z_0\Big\langle L^0 \frac{\varphi^0}{\twp_x} v, \frac{\varphi^0}{\twp_x} v \Big\rangle_{\rho^0} - \Big\langle f \Big( 1-\frac{\varphi^0}{\twp_x}\Big) v, v\Big\rangle_{\m}\, , 
\end{align*} 
it follows that 
\begin{align*}
\langle L v, v\rangle_{\m}
& \leq -\kappa_0  Z_0\Big\| v \frac{\varphi^0}{\twp_x} \Big\|_{\rho^0}^2 + \kappa_0 Z_0\Big\langle v \frac{\varphi^0}{\twp_x}, \varphi^0 \Big\rangle_{\rho^0}^2 - \Big\langle f \Big( 1-\frac{\varphi^0}{\twp_x}\Big) v, v\Big\rangle_{\m}\\
& = - \kappa_0 \int v^2 \frac{\varphi^0}{\twp_x} \m dx + \frac{\kappa_0}{Z^0} \bigg(\int v\varphi^0 \m dx \bigg)^2 - \int f \Big( 1-\frac{\varphi^0}{\twp_x}\Big) v^2 \m dx\\
& = - \kappa_0 \|v\|_{\m}^2 + \int v^2 \Big( 1-\frac{\varphi^0}{\twp_x}\Big) (\kappa_0-f) \m dx + \frac{\kappa_0}{Z^0} \bigg( \int v \varphi^0 \m dx\bigg)^2.
\end{align*}
Now 
\begin{align*}
\int v^2 \Big( 1-\frac{\varphi^0}{\twp_x}\Big) (\kappa_0-f) \m dx 
& \leq \Big\|\Big( 1-\frac{\varphi^0}{\twp_x} \Big)(\kappa_0-f)\Big\|_{\infty} \|v\|_{\m}^2
= \Big\| c \frac{\twp_{xx}}{\twp_x f} (\kappa_0-f)\Big\|_{\infty} \|v\|^2_{\m}
\end{align*}
and
\begin{align*}
	\bigg( \int v \varphi^0 \m dx\bigg)^2
	& = \Big( \int v \Big(\twp_x - c\frac{\twp_{xx}}{f}\Big) \m dx\Big)^2 \\
	& \leq 2 \langle v, \twp_x\rangle_{\m}^2 + 2c^2 \big\langle v, \frac{\twp_{xx}}{f}\big\rangle_{\m}^2 \leq  2 \langle v, \twp_x\rangle_{\m}^2 + 2c^2  \|v\|_{\m}^2 \Big\|\frac{\twp_{xx}}{f}\Big\|_{\m}^2.
\end{align*}
It follows that
\[\langle L v, v \rangle_{\m}  \leq - \kappa(c) \|v\|_{\m}^2 + 2\frac{\kappa_0}{Z^0}\langle v, \twp_x\rangle_{\m}^2, \]
where
\[\kappa(c) = \kappa_0 \Big( 1-\frac{2c^2}{Z^0} \Big\| \frac{\twp_{xx}}{f}\Big\|_{\m}^2 \Big) - |c| \Big\| \frac{\twp_{xx}}{\twp_x f} (\kappa_0-f)\Big\|_{\infty}.\]
Note that $\kappa(c) \xrightarrow{c \rightarrow 0} \kappa_0>0$.
Set 
\begin{equation}
\label{eq:c*}
c^* = \min \{ |c|: \kappa(c) \leq 0 \}.
\end{equation} 
Then (\ref{eq:spectralgapsmallc}) is satisfied with $\kappa= \kappa(c)$ if $|c|\leq c^*$.
\end{proof}

Note that $c$, $\kappa(c)$, $c^*$ are usually unknown variables depending on $w, f, q, r$. 
It is a priori not clear that there exists a setting in which Theorem \ref{thm:spectralgapsmallc} applies. This can be clarified in the neural field example (\ref{eq:nfe}). Consider the neural field traveling wave operator
\[Lv = -v + w\ast(F'(\twp)v)\]
for some kernel $w$ satisfying $M:=\big\|\frac{w_x}{w}\|_{\infty} < \infty$ and some gain function $F$ and the corresponding traveling wave $(\twp,c)$. We define an associated standing wave in the following way. Set $\twp^0 = w\ast F(\twp)$ and $F^0(x) = F(\twp((\twp^0)^{-1}(x)))$ (since $\twp^0$ is increasing, $(\twp^0)^{-1}$ is well-defined). Then $\twp^0 = w\ast F^0(\twp^0)$ is the traveling wave solution to the neural field equation with kernel $w$ and gain function $F^0$, and $\twp^0_x$ is the eigenfunction to the eigenvalue $0$ of $L^0$, where
\[L^0v = -v + w\ast((F^0)'(\twp^0)v) = -v + w\ast \Big(F'(\twp)\frac{\twp_x}{\twp^0_x}v\Big).\]
(Note that, in the notation of the proof of Theorem \ref{thm:spectralgapsmallc}, $\twp^0_x = \varphi^0$.) \sloppy{Since ${\twp^0 = w\ast F(\twp) = \twp-c\twp_x}$, we have}
\begin{equation}
\label{eq:perturbtwp0}
\twp(x) = (I-c\partial_x)^{-1} \twp^0 = \int_0^{\infty} e^{-s} \twp^0(x+cs) ds.
\end{equation}
In this setting, Theorem \ref{thm:spectralgapsmallc} therefore tells us the following. Assume that $L^0$ satisfies a spectral gap inequality in $L^2(\m^0)$ with constant $\kappa_0$.
Set
\[\kappa(c) = \kappa_0 \Big( 1-\frac{2c^2}{Z^0}\| \twp_{xx}\|_{\m}^2 \Big) - |c| \Big\| \frac{\twp_{xx}}{\twp_x} (\kappa_0-1)\Big\|_{\infty}.\]
Since $\big\|\frac{w_x}{w}\big\|_{\infty}=M$ and
\[|\twp_{xx}(x)| = \bigg| \int_0^{\infty} e^{-s} w_x\ast(F'(\twp)\twp_x)(x+cs) ds \bigg| \leq M \twp_x(x),\]
it follows that 
\[\frac{\| \twp_{xx}\|_{\m}^2}{Z^0} = \frac{\int \twp_{xx}^2 F'(\twp) dx}{\int \twp^0_x \twp_x F'(\twp) dx} \leq \frac{M^2}{1-|c|M}\]
and 
\[\kappa(c) \geq \kappa_0 \Big( 1-\frac{2c^2M^2}{1-|c|M}\Big) - |c|M|\kappa_0-1|.\]
Then for all $c$ satisfying
\begin{equation}
\label{eq:smallccondtw}
\kappa_0 \Big( 1-\frac{2c^2M^2}{1-|c|M}\Big) - |c|M|\kappa_0-1|>0,
\end{equation}
the traveling wave operator associated with $\twp$ as defined in (\ref{eq:perturbtwp0}) (that is, the operator with kernel $w$ and gain function $F(x) = F^0(\twp^0(\twp^{-1}(x)))$), satisfies a spectral gap inequality in $L^2(\m)$.

The wave speed $c$ is usually unknown. It would be desirable to express the smallness condition on $c$ in terms of the parameters of the system. As an example, we consider in the next subsection the neural field equation (\ref{eq:nfe}) with synaptic connections described by a two-sided exponential kernel, $w(x) = \frac{1}{2\sigma} e^{-\frac{|x|}{\sigma}}$ for some $\sigma>0$. It is possible to explicitly bound $\kappa_0$ in terms of $F, \sigma, $ and $c$, see \cite{langthesis} for details. 
Furthermore, we can derive bounds on the wave speed in terms of $w$ and $F$ (see Proposition \ref{prop:boundsc}). Together, this allows to translate the smallness condition on $c$ into a condition on the parameters of the system.

\subsection{Bounds on the Wave Speed}
\label{subsection:boundsc}

In \cite{ermentroutmcleod}, Thm. 3.1, Ermentrout and McLeod proved that
\begin{equation}
\label{eq:c}
c = \frac{\int_{a_1}^{a_2}x-F(x) dx}{\int \twp_x^2(x) F'(\twp(x)) dx},
\end{equation}
where $a_1=F(a_1)$ and $a_2=F(a_2)$ are the two stable fixed points of the neural field equation.
We can use this representation to deduce the following lower and an upper bound on $c$.

\begin{proposition}
\label{prop:boundsc}
Assume that $F$ is convex-concave, that is, there exists $z$ such that $F''(x) \geq 0$ for $x\leq z$ and $F''(x) \leq 0 $ for $x\geq z$.
Then the wave speed is bounded in terms of the parameters of the system, $\sigma$ and $F$:
\[\frac{\sigma}{\sqrt{2}(a_2 - a_1)} \frac{\int_{a_1}^{a_2} x-F(x) dx}{\sqrt{\int_{a_1}^a x-F(x) dx}} \leq c \leq \frac{\sigma}{4} \frac{\int_{a_1}^{a_2} x-F(x) dx}{\int_a^{a_2} F(x)-x dx} \]
\end{proposition}
\begin{proof}
	Since $F$ is convex-concave, it can be proven as in Lemma 4.1 in \cite{langMSA} that there exists a unique $x_0$ such that $\twp_{xx}(x_0)=0$ and $\twp_{xx}(x)\geq 0$ for $x\leq x_0$ and $\twp_{xx}(x)\leq 0$ for $x\geq x_0$ (see \cite{langthesis} for details).
	
	We first prove the upper bound. Since $|\twp_{xx}|\leq \frac{1}{\sigma} \twp_x$, we have that
	\[\twp_x(x) \geq e^{- \frac{|x-x_0|}{\sigma}} \twp_x(x_0).\]
	This implies that
	\begin{align*}
		\int \twp_x^2(x) F'(\twp(x)) dx & \geq \int e^{-\frac{|x-x_0|}{\sigma}} \twp_x(x) \twp_x(x_0) F'(\twp(x)) dx \\
		&  = 2 \sigma w*(F'(\twp)\twp_x)(x_0) \twp_x(x_0)  = 2 \sigma (\twp_x(x_0) -c\twp_{xx}(x_0)) \twp_x(x_0) \\
		& =  2\sigma \twp_x^2(x_0).
	\end{align*}
	Since $\sigma^2 \twp_{xx} = \twp - \int_0^{\infty} e^{-s} F(\twp(\cdot+cs)) ds$ we obtain
	\begin{align*}
		\twp_x^2(x) &= 2 \int_x^{\infty} -\twp_{xx}(y) \twp_x(y) dy  = 2 \int_x^{\infty} - \frac{\twp(y) - \int_0^{\infty} e^{-s} F(\twp(y+cs) ds}{\sigma^2} \twp_x(y) dy \\
		& \geq 2 \int_x^{\infty} - \frac{\twp(y)-F(\twp(y))}{\sigma^2} \twp_x(y) dy = \frac{2}{\sigma^2} \int_{\twp(x)}^{a_2} F(x)-x \ dx.
	\end{align*}
	Therefore
	\[\twp_x^2(x_0) = \max \twp_x^2(x) \geq \twp_x^2(\twp^{-1}(a)) \geq \frac{2}{\sigma^2} \int_a^{a_2} F(x) - x \ dx.\]
	Using (\ref{eq:c}), we obtain
	\[ c = \frac{\int_{a_1}^{a_2}x-F(x) dx}{\int \twp_x^2(x) F'(\twp(x)) dx} \leq \frac{1}{2\sigma} \frac{\int_{a_1}^{a_2}x-F(x) dx}{\twp_x^2(x_0)} \leq \frac{\sigma}4 \frac{\int_{a_1}^{a_2}x-F(x) dx}{\int_a^{a_2} F(x)-x dx}.\]
	This yields the upper bound.
	
	We now prove the lower bound.
	We have
	\begin{align*}
		\twp_x^2(x) & = 2 \int_{-\infty}^x \twp_x(y) \twp_{xx}(y) dy = \frac{2}{\sigma^2} \int_{-\infty}^x \left(\twp(y) - \int_0^{\infty} e^{-s} F(\twp(y+cs)) ds \right) \twp_x(y) dy \\
		& \leq \frac{2}{\sigma^2} \int_{-\infty}^x (\twp(y) - F(\twp(y))) \twp_x(y) dy  = \frac{2}{\sigma^2} \int_{a_1}^{\twp(x)} x-F(x) dx.
	\end{align*}
	Since
	\[0=\twp_{xx}(x_0) = \frac{1}{\sigma^2} \left(\twp(x_0) - \int_0^{\infty} e^{-s} F(\twp(x_0+cs)) ds\right) \]
	it follows that 
	\[\twp(x_0) = \int_0^{\infty} e^{-s} F(\twp(x_0+cs)) ds > F(\twp(x_0))\]
	and hence $ \twp(x_0) < a$, so that
	\[ \twp_x^2(x_0) \leq \frac{2}{\sigma^2} \int_{a_1}^{\twp(x_0)} x-F(x) dx \leq \frac{2}{\sigma^2} \int_{a_1}^a x-F(x) dx.\]
	Thus, using (\ref{eq:c}),
	\[c \geq \frac{\int_{a_1}^{a_2} x-F(x) dx}{\twp_x(x_0)\int F'(\twp(x)) \twp_x(x)  dx} \geq \frac{\sigma}{\sqrt{2}} \frac{\int_{a_1}^{a_2} x-F(x) dx}{\sqrt{\int_{a_1}^a x-F(x) dx}\, (F(a_2) - F(a_1))}\, , \]
which implies the lower bound, since $F(a_2) = a_2$ and $F(a_1) = a_1$. 
\end{proof}

\subsection{Stochastic Long-term Stability}
\label{subsection:longtermstab}

In this subsection we will stick to the neural field example. We show how the $L^2(\m)$-spectral gap inequality can be used to derive a long-term stochastic stability result. 
We consider the stochastic neural field equation
\begin{equation} 
\label{StochNFEq} 
du(x,t) = \big(-u(x,t) + w\ast F(u(\cdot,t))(x) \big) dt + \Sigma (u(t))\, dW(x,t), 
\end{equation} 
where $W$ is a cylindrical Wiener process with values in the Hilbert space $L^2$, 
defined on some underlying filtered probability space $(\Omega , \mathcal F, (\mathcal 
F(t))_{t\ge 0}, P)$ (see the monograph \cite{daprato, prevotroeckner}). 
The dispersion coefficient $\Sigma$, describing the standard deviation of the noise term, is 
assumed to be a function of the $L^2$-distance $\inf_{C\in\mathbb R} \|u-\twp (\cdot + C)\|$ of 
$u$ to the set $\mathcal N$ of traveling waves, i.e., $\Sigma (u) = \Sigma (u-\twp (\cdot + C))$ 
for all $C\in\Bbb R$.

\medskip 
\noindent 
A rigorous meaning to equation \eqref{StochNFEq} is given by decomposing $u(t) = v(t) + \twp (t)$ 
w.r.t. the traveling wave. The stochastic evolution equation for $v$ is then given by 
\begin{equation} 
\label{StochRelNFEq} 
\begin{aligned} 
dv(x,t) & = \left( - v(x,t) + w\ast \left( F(v (\cdot , t) + \twp (\cdot ,t))  
 - F(\twp (\cdot , t)) \right)\right)\, dt   + \Sigma (u (t))\, dW(t) 
\end{aligned} 
\end{equation} 
and we will now make the following assumptions on $\Sigma$. $\Sigma : L^2 \mapsto L_2 (L^2, L^2)$ 
satisfies 
\begin{equation}
\label{invariance} 
\Sigma (v) = \Sigma (v + \twp  - \twp (\cdot + C)) \mbox{ for any } C\in\Bbb R \, , \Sigma (0) = 0 
\end{equation}   
and $\Sigma$ is Lipschitz continuous 
\begin{equation}
\label{noise}
\|\Sigma (v_1 ) - \Sigma (v_2)\|^2_{L_2 (L^2,L^2)}\le \sigma^2\|v_1 - v_2\|^2\, . 
\end{equation} 
for some constant $\sigma^2$. Here, $L_2 (L^2,L^2)$ denotes the space of all Hilbert-Schmidt 
operators $L: L^2\mapsto L^2$ 
with the Hilbert-Schmidt norm $\|L\|_{L_2 (L^2,L^2)}^2 = \sum_k \|Le_k\|^2_H$ for one (hence any) 
complete orthonormal system $(e_k)_k$ of $L^2$. 

\medskip 
\noindent 
Standard theory on stochastic evolution equations now implies the existence and uniqueness of 
strong solutions $v(x,t)$ of \eqref{StochRelNFEq} for arbitrary initial condition $v_0\in L^2$, 
so that $u(x,t) = v(x,t) + \twp (x,t)$ is a strong solution of \eqref{StochNFEq}. 

\medskip 
\noindent 
As in \cite{kruegerstannat, langMSA}, we account for shifts in the phase of the wave by dynamically adapting the speed of a reference wave according to
\[
\dot{C}^m(t) = -m \langle u-\twp(\cdot-ct-C^m(t)), \twp_x(\cdot-ct-C^m(t))\rangle_{\m_t}, \qquad C^m(0) = 0, 
\]
where $m$ is a parameter determining the rate of relaxation to the right phase, and where 
\[
\m_t(x) = \m(x-ct-C^m(t)). 
\]
Here we move the measure with the wave such that $\|\twp (x-ct-C^m(t))\|_{\m_t} = \|\twp\|_{\m}$ 
for all $t\geq 0$. Set $\tilde{u}(x,t) = \twp(x-ct-C^m(t))$. $\tilde{v}:= u(x,t) - \tilde{u}
(x,t)$ satisfies
\[
d\tilde{v}(x,t) = \big( L_t \tilde{v}(x,t) + R(t,\tilde{v}) + \dot{C}^m(t) \twp_x(x-ct-C^m(t))\big) dt + \Sigma (\tilde{v}) dW(x,t),
\]
where $L_t$ is the family of time-dependent uniformly bounded operators 
\[
L_t v = -v + w\ast (F'(\twp(\cdot-ct))v), \qquad D(L_t) = L^2, 
\]
and where the rest term is given as 
\[
R(t,v) = w\ast F(\twp(\cdot-ct)+v)-F(\twp(\cdot-ct)) - w\ast (F'(\twp(\cdot -ct)) v). 
\]
By Taylor's formula, there exists $\xi(y,t)$ such that
\[
R(t,v) = \frac{1}{2} \int w(x-y) \Big( F''(\twp(y-ct)+\xi(y,t)) v^2(y)\Big) dy  
\]
and Cauchy-Schwarz inequality implies that 
\begin{equation}
\label{eq:estimateRtilde}
\begin{split}
\|R(t,v)\|_{\m_t}^2
& \leq \frac{1}{4} \|w\|_{\infty} \|F''\|_{\infty}^2 \|v\|^2 \int \int w(x-y) v^2(y) dy\,  \m_t (x) dx \\
& \leq \frac{1}{4} \|w\|_{\infty} \|F''\|_{\infty}^2  \frac{\|\m\|_{\infty}}{\min |\m|^2} \|v\|_{\m_t}^4 =: M_R^2 \|v\|_{\m_t}^4.
\end{split}
\end{equation}
The rest term $R(t,v)$ is thus of higher order in $\|v\|$. It can therefore be expected that the stability properties of the traveling wave depend only on the linear operator $L$.

\medskip 
\noindent 
Note that we cannot expect to have the same control over the rest term in $L^2(\rho)$. As stated in \cite{langMSA}, if $c>0$, then typically there exists $L_{\rho}>0$ such that $\rho(y) \leq L_{\rho} \rho(x)$ for $y\leq x$, and $\lim_{x\rightarrow -\infty} \rho(x) = 0$. Now assume that there exists $C>0$ such that for $v\in L^2(\rho)$,
\[\|w\ast v^2\|^2_{\rho} = \int \int \int w(x-y_1) w(x-y_2) v^2(y_1) v^2(y_2) \rho(x) dy_1 dy_2 dx \leq C \|v\|^4_{\rho}. \]
Formally, letting $v^2\rightarrow  \delta_y$ we obtain
\[\int w^2(x-y) \rho(x) dx \leq C \rho^2(y).\]
But 
\[\frac{\int w^2(x-y) \rho(x) dx }{\rho^2(y)} \geq \frac{\int_y^{\infty} w^2(x-y) dx \rho(y)}{L_{\rho}\rho^2(y)} = \frac{\int_0^{\infty} w^2(x) dx}{L_{\rho}\rho(y)} \xrightarrow{y\rightarrow -\infty} \infty,\]
which is a contradiction.

If the traveling wave operator $L=L_0$ satisfies a spectral gap inequality in $L^2(\m)$, then we can derive a time-uniform bound for $\tilde{v}(t)$. 

\begin{theorem}
\label{thm:stabilityL2m}
Assume that the traveling wave operator $L$ satisfies a spectral gap inequality in $L^2(\m)$,
\[\langle Lv, v\rangle_{\m} \leq - \kappa \|v\|_{\m}^2 + Z \langle v, \twp_x \rangle_{\m}^2\]
for some $\kappa, Z >0$.
Assume that 
$\|c \frac{\m_x}{\m}\|_{\infty} + \sigma^2 \|\m\|_\infty \|\m^{-1}\|_\infty < \kappa$ and $m>Z$.
Set $b^* = \frac{\kappa}{2 M_R + m\|\twp_x\|_{\m}\|\frac{\m_x}{\m}\|_{\infty} }$ and
\[\tau:= \inf \left\{ t\geq 0: \|\tilde{v}(t)\|_{\m_t} \geq b^* \right\}.\]
Then 
\[
P(\tau<\infty) \leq  \frac{1}{(b^*)^2} \|\tilde{v}(0)\|^2_{\m}  . 
\]
\end{theorem} 

\begin{proof}
For $t \leq \tau$, by It\^{o}'s Lemma,
\begin{align*}
d\|\tilde{v}(t)\|_{\m_t}^2 
& = 2 \langle -\tilde{v}(t) + w\ast (F'(\tilde{u}(t))\tilde{v}(t)) + R(t,\tilde{v}(t)), \tilde{v}(t)\rangle_{\m_t} dt + \Big(2\dot{C}^m(t) \langle \tilde{v}(t), \partial_x \tilde{u}(t)\rangle_{\m_t}\\
& \qquad -  (c+\dot{C}^m(t))\int \tilde{v}^2(t) \partial_x \m_t dx  +  \|\Sigma (\tilde{v}(t))\|^2_{L_2(L^2, L^2(\m_t))}\Big) dt\\
& \qquad + 2 \langle \tilde{v}(t), \Sigma (\tilde{v} (t)) dW_t\rangle_{\m_t},
\end{align*}
where we denote by $\|\Sigma (\tilde{v})\|_{L_2(L^2, L^2(\m_t))}$ the Hilbert-Schmidt norm of $\Sigma (\tilde{v})$: for an orthonormal basis $(e_k)$ of $L^2$,
\[ \|\Sigma (\tilde{v})\|_{L_2(L^2, L^2(\m_t))}^2 = \sum_k \|\Sigma (\tilde{v}) e_k\|_{\m_t}^2 \leq \|\m\|_{\infty} \sigma^2 \|\tilde{v}\|^2 \leq \|\m\|_{\infty} \|\m^{-1}\|_\infty 
\sigma^2 \|\tilde{v}\|^2_{\m_t} .
\]
Since $m>Z$, 
\begin{align*}
& \langle -\tilde{v}(t) + w\ast (F'(\tilde{u}(t))\tilde{v}(t)) , \tilde{v}(t)\rangle_{\m_t} + \dot{C}^m(t) \langle \tilde{v}(t), \partial_x \tilde{u}(t)\rangle_{\m_t} \\
& \hspace{2cm} \leq -\kappa \|\tilde{v}(t)\|_{\m_t}^2 + (Z-m) \langle \tilde{v}(t), \partial_x \tilde{u}(t) \rangle_{\m_t}^2 \leq -\kappa \|\tilde{v}(t)\|_{\m_t}^2.
\end{align*}
Set 
\[
M_t = 2\int_0^t \langle \tilde{v}(s),  \Sigma (\tilde{v}(s))\, dW_s\rangle_{\m_s}.
\]
Using (\ref{eq:estimateRtilde}) and 
\[\bigg|(c+\dot{C}^m(t))\int \tilde{v}^2(t) \partial_x \m_t dx\bigg| \leq (c+ m b^* \|\twp_x\|_{\m}) \Big\|\frac{\m_x}{\m}\Big\|_{\infty} \|\tilde{v}(t)\|_{\m_t}^2, \]
we obtain that
\begin{align*}
d\|\tilde{v}(t)\|_{\m_t}^2 
& \leq \Big( -2\kappa  + \Big(2 M_R + m\|\twp_x\|_{\m} \Big\|\frac{\m_x}{\m}\Big\|_{\infty} \Big) b^* \\
& \qquad + \Big\|c \frac{\m_x}{\m}\Big\|_{\infty}  + \sigma^2\|\m\|_{\infty} \|\m^{-1}\|_{\infty} 
    \Big)\|\tilde{v}(t)\|_{\m_t}^2 dt + dM_t \\
& \leq - \Big(\kappa-\Big\|c \frac{\m_x}{\m}\Big\|_{\infty}  
     - \sigma^2\|\m\|_{\infty}\|\m^{-1}\|_\infty \Big) \|\tilde{v}(t)\|_{\m_t}^2 dt  +  dM_t.
\end{align*}
Set $\tilde{\kappa} := \kappa  - \|c \frac{\m_x}{\m}\|_{\infty} 
   - \sigma^2 \|\m\|_\infty \|\m^{-1}\|_\infty > 0$. 
Applying It\^{o}'s formula to $e^{\tilde{\kappa}t}\|\tilde{v}(t)\|_{\m_t}^2$ we get that
$$
e^{\tilde{\kappa}t}\|\tilde{v}(t)\|_{\m_t}^2 
\leq \|\tilde{v}(0)\|_{\m}^2 + \int_0^t e^{\tilde{\kappa}s} dM_s 
$$
for $t \le \tau$, hence 
$$ 
e^{\tilde{\kappa}\tau\wedge t} \|\tilde{v}(t)\|_{\m_t}^2 
\le \|\tilde{v}(0)\|_{\m}^2  + \int_0^{\tau\wedge t} e^{\tilde{\kappa} s} dM_s \, .
$$
Since $\int_0^t e^{\tilde{\kappa}s} dM_s$ is a continuous (local) martingale and square 
integrable up to the stopping time $\tau$, the optional sampling theorem now yields  
$E\big(\int_0^{\tau\wedge t} e^{\tilde{\kappa}s} dM_s\big) = 0$ and thus  
$$ 
E\big( e^{\tilde{\kappa}\tau\wedge t} \|\tilde{v}(t)\|_{\m_t}^2 \big) 
\le  \|\tilde{v}(0)\|_{\m}^2 \, .  
$$
Taking the limit $t\to\infty$ we finally arrive at 
\begin{align*} 
(b^\ast )^2 P (\tau < \infty ) 
& \le \lim_{t\to\infty} E\big( e^{\tilde{\kappa}\tau \wedge t} \|\tilde{v}(t)\|_{\m_t}^2 \big) 
\le \|\tilde{v}(0)\|_{\m}^2 
\end{align*}
which implies the assertion. 
\end{proof}

\begin{remark}\leavevmode
\begin{enumerate}
	\item The theorem tells us that the difference $\tilde{v}$ between the stochastic solution and the adapted traveling wave stays small uniformly in $t$ on the set $\left\{\tau=\infty\right\}$. The probability of this set can be controlled by the initial difference $\|u(0)-\twp\|$ and the noise amplitude $\sigma$. 
	\item Note that also in Theorem \ref{thm:stabilityL2m} $|c|$ is required to be `small enough' since we assume that $\|c \frac{\m_x}{\m}\|_{\infty} < \kappa$. 
\end{enumerate}
\end{remark}

\bigskip 
\noindent 
{\bf Acknowledgement:} We thank the referee for a very careful reading of the manuscript and for 
constructive and clarifying remarks.


\end{document}